\documentclass[12pt]{article}

\usepackage{amssymb,amsmath,amsthm}

\newcommand{\ovl}{\overline}

\newcommand{\vp}{\varepsilon}
\newcommand{\cl}[1]{{\mathcal{#1}}}
\newcommand{\bb}[1]{{\mathbb{#1}}}

\numberwithin{equation}{section}

\theoremstyle{plain}
\newtheorem{lem}{Lemma}[section]

\newtheorem{thm}[lem]{Theorem}
\newtheorem{cor}[lem]{Corollary}

\theoremstyle{definition}

\theoremstyle{remark}
\newtheorem{rem}[lem]{Remark}

\begin{document}

\title{THE CARPENTER AND SCHUR--HORN PROBLEMS FOR MASAS IN FINITE FACTORS}

\date{\today}

\author{Kenneth J. Dykema \thanks{Partially supported by NSF grant
DMS-0901220}\and Junsheng Fang \thanks{Partially supported by the Fundamental
Research Funds for
the Central Universities of China and NSFC(11071027)}\and Donald W.
Hadwin \and Roger
R. Smith \thanks{Partially supported by NSF grant
DMS-1101403}}

\maketitle

\begin{abstract}
Two classical theorems in matrix theory, due to Schur and Horn, relate the
eigenvalues of a self-adjoint
matrix to the diagonal entries. These have recently been given a formulation in
the setting of operator algebras as the
Schur-Horn problem, where matrix algebras and diagonals are replaced
respectively by finite factors and
maximal abelian self-adjoint subalgebras (masas). There is a special case of the
problem, called the carpenter
problem, which can be stated as follows: for a masa $A$ in a finite factor $M$
with conditional expectation
$\mathbb{E}_A$, can each $x\in A$ with $0\leq x\leq 1$ be expressed as
$\mathbb{E}_A(p)$ for a projection
$p\in M$?

In this paper, we investigate these problems for various masas. We give positive
solutions for the generator
and radial masas in free group factors, and we also solve affirmatively a weaker
form of the
Schur-Horm problem for the
Cartan masa in the hyperfinite factor.
\end{abstract}

\section{Introduction}\label{sec1}

\indent 

Two classical theorems due to Schur \cite{Sch} and Horn \cite{Ho}, which relate
the diagonal entries of an $n\times n$ self-adjoint matrix to its eigenvalues,
have recently been reformulated in the setting of type $\text{II}_1$ factors $M$
with normalized trace $\tau$ \cite{A-K}. A special case of the problem, termed
the carpenter problem in \cite{Ka1,Ka2}, asks whether each element $x$ in a masa
$A\subseteq M$ satisfying $0\le x\le 1$ can be expressed as ${\bb E}_A(p)$ for
some projection $p\in M$. This entails $\tau(x) = \tau(p)$, so the analogous problem
in complex matrix algebras
places a constraint on the value of $\tau(x)$. Subject to this, Horn's
theorem gives a positive solution for matrices.

The $II_1$--factor analogue of the diagonal subalgebra in the $n\times n$ matrices is a maximal abelian (self--adjoint)
subalgebra, called a masa, $A\subseteq M$.
We let ${\bb E}_A$ denote the trace--preserving conditional expectation of $M$ onto $A$.
The carpenter problem in a II$_1$--factor is, given $x\in A$ with $0\le x\le 1$, to find a projection $p\in M$ so that
${\bb E}_A(p)=x$;
this problem remains open.

The Schur--Horn problem for a masa
$A\subseteq M$ may be stated as follows: for a suitable notion of spectral majorization of
$x\in A$ by $z\in M$ (described in Section~\ref{sec5}), does there exist an
element $y\in M$ having the same spectral distribution as $z$ so that
$x= {\bb E}_A(y)$?
In this paper we address these two questions for specific
choices of masas. We give positive solutions to both the carpenter problem and
the Schur--Horn problem when $A$ is either a generator masa or the radial masa
in a free group factor. We also investigate the Cartan masa in the hyperfinite
factor, and obtain a version of the Schur--Horn theorem which is slightly weaker
than the one above.

The paper is organized as follows. In Section~\ref{sec2} we present a technical
result giving a sufficient condition for positive solutions of the carpenter
problem (Lemma~\ref{lem2.1}), and all of our subsequent results are based on
this. The main results on masas in free group factors are contained in
Section~\ref{sec3}, while Section~\ref{sec4} is concerned with the carpenter
problem for the Cartan masa in the hyperfinite factor. Here our results are less
definitive, although we do present classes of elements in $A$ for which a
positive solution can be given. In a different direction, we also solve the
carpenter problem for all elements of the Cartan masa $A$, but modulo an
automorphism of $A$.

In the final section, we consider the Schur--Horn problem.
We first consider a minor reformulation of Arveson and Kadison's version of the problem
and show that it is equivalent to theirs.
Then we give a positive
solution for the generator masa and the radial masa in free group factors. We
also investigate the Cartan masa, proving a weaker version of the Schur--Horn
problem as mentioned above.

There has been considerable recent interest in these problems, and we have drawn
heavily on the ideas and results presented in \cite{A-M0,A-M,A-K,Ka1,Ka2}.

\section{An existence method}\label{sec2}

\indent 

In the first lemma below we will describe a sufficient condition for solving the carpenter problem positively, and in subsequent sections we will apply it in various situations.

We fix a finite von Neumann algebra $M$ with a normal normalized trace $\tau$ and a masa $A\subseteq M$. We denote the unique trace preserving conditional expectation of $M$ onto $A$ by ${\bb E}_A$. For each $x\in A$ satisfying $0\le x \le 1$, we introduce the $w^*$-compact convex subset $\Gamma_x \subseteq M$, defined by
\begin{equation}\label{eq2.0}
\Gamma_x = \{y\in M\colon \ 0\le y\le 1, \quad {\bb E}_A(y) = x\}.
\end{equation}
This set is nonempty since it contains $x$, and any projection $p\in\Gamma_x$ is a solution of the carpenter problem for the element $x\in A$. Any such projection is automatically an extreme point of $\Gamma_x$, and so it suffices to consider the extreme points of $\Gamma_x$. These are abundant, by the Krein--Milman theorem.

For each nonzero projection $e\in M$, define a bounded map $\Phi_e\colon \ eMe\to A$ by
\begin{equation}\label{eq2.1}
\Phi_e(ex e) = {\bb E}_A(ex e), \qquad x\in M.
\end{equation}

\begin{lem}\label{lem2.1}
Let $A$ be a masa in a finite von Neumann algebra $M$, and suppose that $\Phi_e$ is not injective for each nonzero projection $e\in M$. Given $x\in A$ satisfying $0\le x\le 1$, there exists a projection $p\in M$ such that ${\bb E}_A(p) =x$.
\end{lem}

\begin{proof}
Fix an arbitrary $x\in A$ satisfying $0\le x\le 1$. Under the stated hypotheses,
we will show that every extreme point of $\Gamma_x$ is a projection and the
result then follows. To obtain a contradiction, let $y$ be an extreme point of
$\Gamma_x$ which is not a projection. For a sufficiently small choice of
$\vp>0$, the spectral projection $e$ of $y$ for the interval $(\vp,1-\vp)$ is
nonzero. Since $\Phi_e$ is not injective we may choose a nonzero element $z\in
eMe$ so that $\Phi_e(z)=0$. By considering real and imaginary parts we may take
$z$ to be self-adjoint, and by scaling we may assume that $\|z\|\le \vp$.
Note that ${\bb E}_A(y\pm z)={\bb E}_A(y)=x$. Since $\vp e\le ye\le (1-\vp)e$,
it follows that $0\le y \pm z\le 1$, and so $y\pm z\in \Gamma_x$ with $y =
((y+z) +(y-z))/2$. This contradicts the assumption that $y$ is an extreme point,
showing that every extreme point is a projection.
\end{proof}

To illustrate the use of Lemma~\ref{lem2.1}, we now show that the carpenter problem has a positive solution for any masa in a free group factor with an uncountable number of generators.

\begin{thm}\label{thm2.2}
Let $S$ be an uncountable set and let $ {\bb F}_S$ be the free group on a set of generators indexed by $S$. If $A$ is a masa in $L({\bb F}_S)$ and $x\in A$ satisfies $0\le x\le 1$, then there exists a projection $p\in L({\bb F}_S)$ such that ${\bb E}_A(p)=x$.
\end{thm}

\begin{proof}
From \cite{Po1}, any masa $A$ in $L({\bb F}_S)$ is separable as a von Neumann
algebra. Cardinality considerations then show that $\Phi_e$ must have a
nontrivial kernel for each nonzero projection $e\in L({\bb F}_S)$, and the
result follows from Lemma~\ref{lem2.1}.
\end{proof}

\begin{rem}\label{rem2.3}
\noindent (i)\quad The maps $\Phi_e$ introduced above are normal and so have
preduals. It is an easy calculation to see that $(\Phi_e)_*\colon  \ L^1(A) \to
L^1(eMe)$ is given by $(\Phi_e)_*(a) = eae$, $a\in A$, and extended by
continuity to $L^1(A)$. It then follows that noninjectivity of $\Phi_e$ is
equivalent to the failure of $eAe$ to be $\|\cdot\|_1$-dense in $eMe$, a
potentially useful reformulation.

\medskip

\noindent (ii) \quad In the case of type $\text{II}_1$ factors, we have no
example of a nonzero projection $e$ for which $\Phi_e$ is injective. However,
this can
occur for type I factors. Take $A$ to be the diagonal masa in $B(H)$ and let
$e\in A$ be a rank one projection. Then $eB(H)e = eAe$ and $\Phi_e$ is injective
in this case.

\medskip

\noindent (iii) \quad If $e\in M$ is a projection such that $e\{e,A\}'' e\ne
eMe$ then the map $\Phi_e$ is not injective. To see this, let $N = \{e,A\}''$
and observe that the condition $eNe\ne eMe$ gives a nonzero element $exe\in eMe$
so that ${\bb E}_{eNe}(exe)=0$. Then
\begin{equation}\label{eq2.a}
{\bb E}_A(exe) = {\bb E}_A({\bb E}_N(exe)) = {\bb E}_A({\bb E}_{eNe}(exe))= 0
\end{equation}
and $\Phi_e$ is not injective.$\hfill\square$
\end{rem}

The third part of this remark leads to a connection with another open problem, the question of whether separable von Neumann algebras must be singly generated.

\begin{lem}\label{lem2.4}
Let $M$ be a type $\text{\rm II}_1$ factor and let $A$ be a separable masa. If there exists a nonzero projection $e\in M$ such that $\Phi_e$ is injective, then $M$ is singly generated.
\end{lem}

\begin{proof}
Let $N= \{A,e\}''$ and let $z$ be the central support of $e$ in $N$. Then $z$ is the identity element for the $w^*$-closed ideal $\ovl{NeN}^{w^*}$ in $N$. By Remark~\ref{rem2.3} (iii), the injectivity of $\Phi_e$ implies that  $eNe=eMe$. For any $m\in M$,
\begin{equation}\label{eq2.b}
zmz \in \ovl{NeN~m~NeN}^{w^*} \subseteq \ovl{NeMeN}^{w^*} = \ovl{NeNeN}^{w^*} =
Nz,
\end{equation}
showing that $zMz \subseteq zNz$. The reverse containment is obvious and so $zMz = Nz$. Since $z\in A$, this gives $zMz = \{Az,e\}''$, so the separability of $A$ implies that $zMz$ is generated by two self-adjoint elements $x_1$ and $x_2$. By adding a multiple of $z$ and scaling, we may assume that $0\le x_1\le z$.

Since $M$ is a finite factor, we can find projections $z_2,\ldots, z_n\in M$ which are all equivalent to subprojections of $z$ and such that $z + \sum\limits^n_{i=2} z_i=1$. Then choose partial isometries $v_2,\ldots, v_n\in M$ so that $v^*_iv_i=z_i$ and $v_iv^*_i\le z$ for $2\le i\le n$, and define
\begin{equation}
y_1=x_1+2z_2+\cdots+ nz_n, \quad y_2=e +v_2+v^*_2 +\cdots+ v_n+v^*_n.
\end{equation}\label{eq2.bb}
By construction, $z_2,\ldots, z_n$ are spectral projections of $y_1$ and so lie in $\{y_1,y_2\}''$, showing that this algebra also contains $x_1$. Since $y_2z_i=v_i$, $2\le i\le n$, we see that $\{y_1,y_2\}''$ also contains $v_2,\ldots, v_n$ and $e$, so in particular $zMz \subseteq \{y_1,y_2\}''$. Now $v^*_izv_i=z_i$, and so $z_iMz_j\subseteq \{y_1,y_2\}''$, showing that $M = \{y_1,y_2\}''$. Thus $M$ is singly generated by $y_1+iy_2$.
\end{proof}

It is currently unknown whether separable type $\text{II}_1$ factors exist that are not singly generated. Lemmas~\ref{lem2.1} and \ref{lem2.4} show that any such example would have a positive solution to the carpenter problem for any masa $A$.

We conclude this section by presenting a class of masas for which the carpenter problem has a positive solution. We will need a preliminary lemma which gives a norm density result. 

\begin{lem}\label{lem2.5}
Let $M$ be a separable type $\text{\rm II}_1$ factor and let $A$ be a masa in
$M$.
\begin{itemize}
\item[\rm (i)] If $r\in {\bb Q}\cap [0,1]$ then there exists a projection $p\in M$ so that ${\bb E}_A(p) = r1$. 
\item[\rm (ii)] Given $\vp>0$ and $x\in A$ satisfying $0\le x\le 1$, there exists a projection $p\in M$ such that
$
\|x-{\bb E}_A(p)\|<\vp.
$
\end{itemize}
\end{lem}

\begin{proof}
(i) \quad  The cases $r=0$ and $r=1$ are trivial so we may assume that
$r=k/n$ where $1\le k\le n-1$ for integers $k,n$. In $A$, choose $n$ orthogonal
projections $e_{11},\ldots, e_{nn}$ of trace $1/n$ and choose a matrix algebra
${\bb M}_n\subseteq M$ with diagonal ${\bb D}_n$ so that the $e_{ii}$'s are the
minimal diagonal projections. Since the $e_{ii}$'s lie in ${\bb D}_n$, the two
conditional expectations ${\bb E}_A$ and ${\bb E}_{{\bb D}_n}$ agree on ${\bb
M}_n$. From \cite{Ho}, there is a projection $p\in {\bb M}_n\subseteq M$ so that
${\bb E}_{{\bb D}_n}(p) = (k/n)I_n$, and so ${\bb E}_A(p) = (k/n)1\in A$.

\medskip

\noindent (ii) \quad Now consider a fixed but arbitrary $x\in A$ satisfying
$0\le x\le 1$ and let $\vp>0$ be given. Since $A$ is separable we may identify
$A$ with $L^\infty[0,1]$ and then we may choose projections $e_k$, $1\le k\le
n$, summing to 1, corresponding to disjoint measurable subsets of [0,1], and
constants $\lambda_k \in [0,1]$ so that
\begin{equation}\label{eq2.c}
\left\|x-\sum^n_{k=1} \lambda_ke_k\right\|_\infty <\vp.
\end{equation}
A further approximation allows us to assume that each $\lambda_k$ is rational in
[0,1]. Applying (i) to the containment $Ae_k \subseteq e_kMe_k$, we find
projections $p_k\le e_k$, $1\le k\le n$, so that ${\bb E}_{Ae_k}(p_k) =
\lambda_ke_k$. If we define a projection by $p = \sum\limits^n_{k=1}p_k$, then
${\bb E}_A(p) = \sum\limits^n_{k=1} \lambda_ke_k$ and $\|x-{\bb E}_A(p)\| <\vp$
as required.
\end{proof}

\begin{thm}\label{thm2.6}
Let $A$ be a masa in a type $\text{\rm II}_1$ factor $M$ and let $\omega$ be a free ultrafilter on ${\bb N}$. Then the carpenter problem has a positive solution for the masa $A^\omega\subseteq M^\omega$.
\end{thm}

\begin{proof}
Let $x\in A^\omega$ satisfy $0\le x\le 1$ and choose a representative
$(x_1,x_2,\ldots)$ for $x$ where $x_n\in A$ and $0\le x_n\le 1-1/n$. By
Lemma~\ref{lem2.5}, there exist elements $y_n\in A$, $0\le y_n\le 1$, and
projections $p_n\in M$ such that $\|x_n-y_n\| <\frac1n$ and ${\bb
E}_A(p_n)=y_n$. Then $(y_1,y_2,\ldots)$ is also a representative for $x$,
$p=(p_1,p_2,\ldots)$ is a projection in $M^\omega$, and it follows that ${\bb
E}_{A^\omega }(p) = ({\bb E}_A(p_1)$, ${\bb E}_A(p_2),\ldots)=x$.
\end{proof}

\section{Free group factors}\label{sec3}

\indent 

In this section we consider the carpenter problem in  free group factors. Let
${\bb F}_n$ denote the free group on $n$ generators $\{g_1,\ldots, g_n\}$, $2\le
n<\infty$. There are types
of masas in the free group factor $L({\bb F}_n)$ that have been much studied.
Each $g_i$ generates a masa
$A_i$, called a generator masa. The second type is the radial or Laplacian
masa, whose generator is the self-adjoint element $\sum\limits^n_{i=1}
(g_i+g^{-1}_i)$. We consider first the generator masa.

\begin{thm}\label{thm3.1}
Let $g_1,\ldots, g_n$ be the generators for ${\bb F}_n$, $2\le n\le\infty$, and let $A_i$ be the $i^{\text{th}}$ generator masa, where $i$ is fixed.
Given $x\in A_i$, $0\le x\le 1$, there exists a projection $p\in L({\bb F}_n)$
such that ${\bb E}_{A_i}(p)=x$.
\end{thm}

\begin{proof}
We first consider the case $n=2$, and without loss of generality we take $i=1$.
Let $S_0$ be an uncountable set and let $S=\{1,2\} \cup S_0$. Then the free
group factor $L({\bb F}_S)$ with generators $g_1,g_2$, and $g_s$ for $s\in S_0$
contains $A_1$ as a masa. By Theorem~\ref{thm2.2}, there is a projection $q\in
L({\bb F}_S)$ such that ${\bb E}_{A_1}(q)=x$. The underlying Hilbert space
$L^2(L({\bb F}_S))$ has an orthonormal basis of group elements and the Fourier
series for $q$ can only have countably many nonzero terms. Thus there is a
countable subset $T\subseteq S$, whose elements we list as $t_1,t_2,\ldots$ with
$t_1=1$, so that $q\in L({\bb F}_T)$. Define an embedding $\phi\colon \ L({\bb
F}_T) \to L({\bb F}_2)$ on generators by $\phi(g_{t_i}) =
g^{i-1}_2g_1g^{1-i}_2$, $i\ge 1$. Then $\phi$ is the identity on $A_1$, and
${\bb E}_{A_1}(\phi(q)) = x$. The desired projection is then $p=\phi(q)$.

For the general case, choose an integer $j\ne i$. Then $A_i \subseteq
L(\{g_i,g_j\})\cong L({\bb F}_2) \subseteq L({\bb F}_n)$, and the result follows
from above since the desired projection can be chosen from $L({\bb F}_2)$.
\end{proof}

For the notion of freeness that is used below, see \cite{V85} or \cite{VDN}.

\begin{cor}\label{cor:freesym}
In a type ${\mathrm{II}}_1$ factor $M$ with tracial state $\tau$, if
$A\subseteq M$ is a masa
and if $s\in M$ is a symmetry with $\tau(s)=0$ and such that $A$ and $\{s\}$ are free
with respect to $\tau$, then for every $x\in A$ satisfying $0\le x\le1$, there exists
a projection $p\in M$ so that ${\bb E}_A(p)=x$.
\end{cor}
\begin{proof}
Since $A$ and $sAs$ are free
and together generate a copy of $L({\bb F}_2)$, this follows from
Theorem~\ref{thm3.1}.
\end{proof}

\begin{rem}\label{rem3.2}
Now it is clear that if a masa
$A\subseteq L({\bb F}_n)$ is supported on at most $n-1$ generators, then 
the carpenter problem for
$A$ has a positive solution.
$\hfill\square$
\end{rem}

We now consider the radial masa $B$ in $L({\bb F}_n)$ for $2\le n<\infty$.

\begin{thm}\label{thm3.3}
Let $B$ be the radial masa in $L({\bb F}_n)$ for a fixed $n$ in the range $2\le n<\infty$. Given $x\in B$, $0\le x\le 1$, there exists a projection $p\in L({\bb F}_n)$ so that ${\bb E}_B(p)=x$.
\end{thm}

\begin{proof}
Let $g_1,\ldots, g_n$ be the generators of ${\bb F}_n$ and let $A_i$ be the
$i^{\text{th}}$ generator masa. For each $i$, let $h_i$ be $g_i+g^{-1}_i$ and
let $L_i\subseteq A_i$ be the abelian von Neumann algebra generated by $h_i$. We
denote by $L$ the von Neumann algebra generated by $\{L_i\colon \ 1\le i\le n\}$
which can be regarded as the free product $L_1*L_2 *\cdots* L_n$.

If we identify $A_i$ with $L^\infty[-1,1]$, then $L_i$ is the subalgebra of even
functions. let $v_i\in A_i$ be the self-adjoint unitary corresponding to the odd
function $1-2\chi_{[0,1]}$. For each $f\in L_i$, $fv_i$ is an odd function and
so has trace 0. We now wish to show that the algebras $v_1Lv_1$,
$v_2Lv_2,\ldots, v_nLv_n$ are free.

Recall that the centered elements of a type $\text{II}_1$ factor $N$ are $\overset{\circ}{N} = \{y\in N\colon \ \tau(y)=0\}$. In order to show freeness, of the algebras $v_iLv_i$, it suffices to show that the trace vanishes on finite products of the form
\begin{equation}\label{eq3.a}
v_{i_1}y_1v_{i_1}v_{i_2}y_2v_{i_2} \ldots v_{i_k}y_kv_{i_k}
\end{equation}
where each $y_i\in \overset{\circ}{L}$ and $i_j\ne i_{j+1}$ for $1\le j\le k-1$. Products of the form $z_{r_1}z_{r_2} \ldots z_{r_s}$ with $z_{r_i}\in \overset{\circ}{L}_{r_i}$, $r_i\ne r_{i=1}$, span a weakly dense subspace of $L$ so we may assume that each $y_i$ has this form. Consider
\begin{equation}\label{eq3.b}
v_{i_1}y_1v_{i_1} = v_{i_1}z_{r_1}\ldots z_{r_s}v_{i_1}.
\end{equation}
A cancellation is only possible if $z_{r_1}\in L_{i_1}\subseteq A_{i_1}$ or
$z_{r_s} \in L_{i_1} \subseteq A_{i_1}$. In the first case, $v_{i_1}z_{r_1}$ is
an element of $A_{i_1}$ and so corresponds to an odd function on $[-1,1]$. Thus
$v_{i_1}z_{r_1}\in \overset{\circ}{A}_{i_1}$, and similarly, if $z_{r_s}\in
L_{i_1}$, then $z_{r_s}v_{i_1} \in \overset{\circ}{A}_{i_1}$. Analyzing in the
same way the behavior when each $v_{i_j}$ is adjacent to a $y_j$ leads to the
conclusion that the element of \eqref{eq3.a} has trace 0. Thus the algebras
$v_1Lv_1,\ldots, v_nLv_n$ are free, implying that $v_1Bv_1,\ldots, v_nBv_n$ are
free. Thus $B$ and $v_1v_2Bv_2v_1$ are free subalgebras of $L({\bb F}_n)$ and
Corollary~\ref{cor:freesym} finishes the proof.
\end{proof}

\section{Crossed products and tensor products}\label{sec4}

\indent

One of the most important masas is the Cartan masa  $A$ in the hyperfinite
$\text{II}_1$ factor $R$. From \cite{CFW}, it is unique up to isomorphisms of
$R$. While the carpenter problem is open in this case, significant progress has
been made in \cite{A-M0,A-M}. In this section we display classes of elements in
$A$ for which a positive solution can be given.

There are many ways of constructing the hyperfinite factor $R$. One that we will employ below is to let ${\bb Z}$ act on $L^\infty({\bb T})$ by irrational rotation, whereupon the crossed product $L^\infty({\bb T}) \rtimes {\bb Z}$ is isomorphic to $R$. In keeping with our earlier techniques, we will enlarge the crossed product and exploit the nonseparability of the resulting algebra.

As a vector space over the field of rationals ${\bb Q}$, the real field ${\bb
R}$ has an uncountable Hamel basis $\{\theta_\alpha\colon \ \alpha\in S\}$,
where $S$ is an uncountable index set. For integers $n_1,\ldots, n_{k+1}$, the
equation $\sum\limits^k_{i=1} n_i\theta_{\alpha_i}=n_{k+1}$ can only be
satisfied by taking all the $n_i$'s to be 0. Then the group $G$, defined to be
the set of all finite  sums $\{n_1\theta_{\alpha_1} +\cdots+
n_k\theta_{\alpha_k} \colon \ n_i\in {\bb Z}, \ \alpha_i\in S\}$ under addition,
can be expressed as $\sum\limits_{\alpha\in S} G_\alpha$, where $G_\alpha =
\{n\theta_\alpha\colon \ n\in {\bb Z}\} \cong {\bb Z}$. The  group $G$ acts on
$L^\infty({\bb T})$ by irrational rotation, and the crossed product
$L^\infty({\bb T})\rtimes G$ is a type $\text{II}_1$ factor and so has a
faithful trace. 

\begin{thm}\label{thm4.1}
Let $R$ be the separable hyperfinite $\text{\rm II}_1$ factor and let $A$ be the Cartan masa in $R$. Given $x\in A$, $0\le x\le 1$, there exists a trace preserving automorphism $\phi$ of $A$ and a projection $p\in R$ such that ${\bb E}_A(p) = \phi(x)$.
\end{thm}

\begin{proof}
Fix $x\in A$, $0\le x\le 1$, and let $\phi_1\colon \ A\to L^\infty({\bb T})$ be an isomorphism that takes the trace on $A$ to integration by Lebesgue measure on ${\bb T}$.
The algebra $L^\infty({\bb T})$ is a separable masa in the nonseparable factor
$L^\infty({\bb T})\rtimes G$, and so by Lemma~\ref{lem2.1} there is a projection
$q\in L^\infty({\bb T})\rtimes G$ so that ${\bb E}_{\phi_1(A)}(q) = \phi_1(x)$.
Elements of $L^\infty({\bb T})\rtimes G$ have Fourier series $\sum\limits_{g\in
G}a_gg$ for $a_g\in L^\infty({\bb T})$, only a countable number of whose terms
are nonzero. Thus there is a countable subgroup $H$ of $G$ so that $q\in
L^\infty({\bb T}) \rtimes H\subseteq L^\infty({\bb T}) \rtimes G$. Now
(after enlarging $H$ if needed to contain an irrational element)
$L^\infty({\bb T})\rtimes H$ is a copy of $R$, and so the uniqueness of Cartan
subalgebras in $R$ gives an isomorphism $\phi_2\colon \ L^\infty({\bb T})
\rtimes H\to R$ so that $\phi_2(L^\infty({\bb T})) =A$, and note that $\phi_2$
is trace preserving. Then
\begin{equation}\label{eq3.d}
{\bb E}_A(\phi_2(q)) = \phi_2({\bb E}_{L^\infty({\bb T})}(q))= \phi_2\phi_1(x).
\end{equation}
Set $p=\phi_2(q)$ and $\phi=\phi_2\phi_1$ to conclude that ${\bb E}_A(p) = \phi(x)$.
\end{proof}

\begin{rem}\label{rem4.2}
(i) \quad If $x$ is taken to be $\lambda1$ for any $\lambda\in [0,1]$, then
there exists a projection $p\in R$ such that ${\bb E}_A(p) =
\phi(\lambda1)=\lambda1$.

\medskip

\noindent (ii)\quad  In Theorem \ref{thm4.1}, an identical proof gives a more
general result: \ given $\{x_i\}^\infty_{i=1} \in A$, $0\le x_i\le 1$, there
exists an automorphism $\phi$ of $A$ and projections $p_i\in R$ so that ${\bb
E}_A(p_i) = \phi(x_i)$, $i\ge 1$.$\hfill\square$
\end{rem}

If $A$ is the Cartan masa in $R$ then, by uniqueness, the inclusions $A\subseteq R$ and $A\ovl\otimes A \subseteq R\ovl\otimes R$ are equivalent.
In the latter formulation we now obtain a large class of elements for which we can solve the carpenter problem.

\begin{thm}\label{thm4.3}
Let $A$ be the Cartan masa in the hyperfinite factor $R$ and let $x\in A$, $0\le x\le 1$. Then there exists a projection $p\in R\ovl\otimes R$ so that ${\bb E}_{A\ovl\otimes A}(p) = x\otimes 1$.
\end{thm}

\begin{proof}
Let $S$ be an uncountable set, and for each $\alpha\in S$ let $A_\alpha$ be a copy of $A$ inside $R_\alpha$, a copy of $R$. Form $M = R\ovl\otimes \big(\underset{\alpha\in  S}{\ovl\bigotimes} R_\alpha\big)$, and denote the abelian subalgebra $A\ovl\otimes \big(\underset{\alpha\in  S}{\ovl\bigotimes} A_\alpha\big)$ by $N$. We identify $R$ with $R\otimes 1\subseteq M$.

Let $e\in M$ be a nonzero projection. Then there exists a countable subset $T\subseteq S$ such that $e$ lies in $R\ovl\otimes \big(\underset{\alpha\in  T}{\ovl\bigotimes} R_\alpha\big)$. Then
\begin{equation}\label{eq3.d1}
eMe = e\left(R\ovl\otimes \left(\big(\underset{\alpha\in  T}{\ovl\bigotimes}
R_\alpha\right)\right) e~\ovl\otimes \left(\underset{\alpha\in S\backslash
T}{\ovl\bigotimes} R_\alpha\right)
\end{equation}
and
\begin{equation}\label{eq3.e}
e\{N,e\}'' e = e \left(\left\{ A\ovl\otimes \left(\underset{\alpha\in 
T}{\ovl\bigotimes} A_\alpha\right), e\right\}''\right) e~\ovl\otimes
\left(\underset{\alpha\in S\backslash T}{\ovl\bigotimes} A_\alpha\right).
\end{equation}
Thus $e\{N,e\}'' e\ne eMe$, and we can find a nonzero element $z\in eMe$ such
that ${\bb E}_{e\{N,e\}''e}(z)= 0$. It follows that ${\bb E}_N(z)=0$, so
Lemma~\ref{lem2.1} applies to give a projection $q\in M$ such that ${\bb E}_N(q)
=  x\otimes 1$. Now $q$ is supported by $R\ovl\otimes \big(\underset{\alpha\in
S_2}{\ovl\bigotimes} R_\alpha\big)$ for a countable subset $S_2$ of $S$. Then
there is an isomorphism $\phi\colon \ R\to \underset{\alpha\in
S_2}{\ovl\bigotimes} R_\alpha$ which maps $A$ to $\underset{\alpha\in
S_2}{\ovl\bigotimes} A_\alpha$, and $\theta = 1\otimes\phi\colon \ R\ovl\otimes
R \to R\otimes \big(\underset{\alpha\in  S_2}{\ovl\bigotimes} R_\alpha\big)$ is
also an isomorphism. If we define a projection $p\in R\ovl\otimes R$ by
$p=\theta^{-1}(q)$, then it follows that ${\bb E}_{A\ovl\otimes A}(p) = x\otimes
1$ as required.
\end{proof}

\section{The Schur--Horn theorem}\label{sec5}

\indent

Let $A$ be a self-adjoint $n\times n$ matrix, let
$\alpha_1\ge\alpha_2\ge\cdots\ge \alpha_n$ be a decreasing rearrangement of the
diagonal entries and let $\lambda_1\ge\lambda_2\ge\cdots\ge \lambda_n$ be a
decreasing ordering of the eigenvalues. A classical theorem of Schur \cite{Sch}
states that
\begin{equation}\label{eq5.1}
\sum^k_{i=1} \alpha_i \le \sum^k_{i=1}\lambda_i,\qquad 1\le k\le n,
\end{equation}
with equality when $k=n$. These inequalities can be used to define a partial
ordering on general $n$-tuples of real numbers by $\alpha \preceq \lambda$ if
\eqref{eq5.1} holds for the decreasing rearrangements of the entries, with
equality when $k=n$. A converse to Schur's theorem was proved by Horn in
\cite{Ho}: if two $n$-tuples $\alpha$ and $\lambda$ satisfy
$\alpha\preceq\lambda$ then there is a self-adjoint matrix $A$ so that the
diagonal is $\alpha$ and the eigenvalues are the entries of $\lambda$.
Collectively, these two results are known as the Schur--Horn theorem.
If we
denote by ${\bb E}_{{\bb D}_n}$ the conditional expectation of ${\bb M}_n$ onto
the diagonal ${\bb D}_n$, then there is an equivalent reformulation of the
Schur--Horn theorem as follows (see \cite{Ka1,Ka2}). If $\alpha$ and $\lambda$
are $n$-tuples of real numbers and $D_\alpha$ is the diagonal matrix with
entries from $\alpha$, then $\alpha \preceq \lambda$ if and only if there exists
a unitary matrix $U\in {\bb M}_n$ so that
\begin{equation}\label{eq5.2}
{\bb E}_{{\bb D}_n}(UD_\lambda U^*) =D_\alpha.
\end{equation}
When $\lambda$ has entries that are 0 or 1, then $D_\lambda$ is a projection,
and \eqref{eq5.2} reduces to a solution of the carpenter problem for the masa
${\bb D}_n$. An appropriate formulation of the Schur--Horn theorem for type
$\text{II}_1$ factors $M$ with a normalized trace $\tau$ was given by Arveson
and Kadison in \cite{A-K} (see also the work of Hiai \cite{Hi,Hi2}), as we now
describe. For each self-adjoint $a\in M$,
the distribution of $a$ is the unique Borel probability measure $m_a$ on ${\bb R}$ so that
\begin{equation}\label{eq5.3}
\int_{\bb R} t^n \ dm_a(t) = \tau(a^n),\qquad n=0,1,2,\ldots~.
\end{equation}
To each Borel subset $B$ of ${\bb R}$, there corresponds a spectral projection
$e_B$ of $a$, and it follows from \eqref{eq5.3} that $m_a(B) = \tau(e_B)$.
Moreover, $m_a$ is supported on the spectrum $\sigma(a)$ of $a$, and is called
the spectral
distribution of $a$. Following \cite[Definition 6.2]{A-K}, we say that a
compactly supported probability measure $n$ on ${\bb R}$ dominates a similar
probability measure $m$ on ${\bb R}$ if 
\begin{align}
\int_{\bb R} t\ dm(t) &= \int_{\bb R} t \ dn(t)\quad \text{and}\notag\\
\label{eq5.4}
\int^\infty_t m((s,\infty))\ ds &\le \int^\infty_t n((s,\infty))
 \ ds,\qquad t\in
{\bb R}.
\end{align}
We deviate slightly from \cite{A-K} which uses closed intervals $[s,\infty)$,
but this makes no difference since $m((s,\infty))$ and $m([s,\infty))$ can only
be unequal on a countable set of $s$-values. For $a,b\in M_{s.a.}$, we can then
define $a\preceq b$ to mean $m_a\preceq m_b$.

This relation $a\preceq b$ can be rewritten to resemble more closely the classical condition~\eqref{eq5.1}
for matrices.
Instead of the eigenvalue sequence of a matrix, for $a=a^*\in M_{s.a.}$, we have the eigenvalue function,
which is the real--valued, monotone nonincreasing, right--continuous function
\[
\mu_t(a)=\inf\{s\in{\bb R}: m_a((s,\infty))\le t\}
\]
of $t\in[0,1)$.
This is the unique real--valued, nonincreasing, right--continuous function
 so that we have $a=\int_0^1\mu_t(a)\,dE(t)$ for some projection--valued
measure $E$ on $[0,1)$ such that $\tau(E([0,t)))=t$ (the actual measure $E$ is obtained by reparameterizing $e_a$).
This eigenvalue function $\mu_t(a)$ is analogous to the decreasing eigenvalue sequence, and we have,
for example,
\[
\tau(a^k)=\int_0^1(\mu_t(a))^k\,dt,\quad(k\ge1).
\]
It is, after a change of variable, the function defined by Murray and von Neumann~\cite[Lemma 15.2.1]{MvN}
and used in various forms by several authors (e.g.\ \cite{Kam}, \cite{Petz}, \cite{BL06}).
In terms of eigenvalue functions, the relation $a\preceq b$ is characterized by the inequalities
\[
\int_0^t\mu_s(a)\,ds \le \int_0^t\mu_s(b)\,ds
\]
for all $0\le t\le 1$ with equality at $t=1$.
This
follows by combining Theorem~2.1 of \cite{A-M0} (which comes from \cite{Hi}) with
Proposition~6.1 of \cite{A-K}.

The analog of Schur's theorem was
established in \cite[Theorem 7.2]{A-K}:

\begin{thm}\label{thm5.1}
If $A$ is a masa in a type $\text{\rm II}_1$ factor $M$, then ${\bb E}_A(x)\preceq x$ for all self-adjoint elements $x\in M$.
\end{thm}

Let ${\cl O}(x)$ denote the norm closure of the unitary orbit of a self-adjoint $x\in M$.
Then $y\in{\cl O}(x)$ if and only if $x$ and $y$ have the same spectral data, i.e., $\mu_t(x)=\mu_t(y)$ for all $t\in[0,1)$
or, equivalently $m_x=m_y$.
This was shown by Kamei~\cite{Kam}, and also in~\cite{A-K}.

The analog of Horn's theorem is then the following problem.
If $A$ is a masa in a type $\text{II}_1$ factor $M$ and $x\in M_{s.a.}$ and $y\in A_{s.a.}$ satisfy $y\preceq x$,
does $y$ lie in ${\bb E}_A({\cl O}(x))$?
In attempting to answer this question, it is unchanged by adding multiples of the identity to $x$ and $y$, and so it suffices to assume that $x,y\ge 0$.
For $x\in M^+$, the eigenvalue function's values $\mu_t(x)$ are actually the generalized $s$--numbers of
\cite{MvN} (see the account in \cite{F-K}), which are defined for
$x\in M$ and $t\ge 0$ by 
\begin{equation}\label{eq5.6}
\mu_t(x) = \inf\{\|xe\|\colon \ e\in {\mathcal P}(M), \ \tau(e)\ge 1-t\},
\end{equation}
where ${\mathcal P}(M)$ denotes the set of projections in $M$.
This was established in \cite[Proposition 2.2]{F-K}.
In particular, we have $\mu_0(x) = \|x\|$.

For $x\in M^+$ we will need the distribution function $\lambda_t(x)$,
defined in \cite[Definition 1.3]{F-K}
to be 
\begin{equation}\label{eq5.5}
\lambda_t(x) = m_x(t,\infty) = \tau(e_{(t,\infty)}(x)),
\end{equation}
so that we have
\begin{equation}\label{eq5.7}
\mu_t(x) = \inf\{s\ge 0\colon \ \lambda_s(x) \le t\}.
\end{equation}
Finally we will need the Ky Fan norms
\begin{equation}\label{eq5.8}
\|x\|_{(t)} = \int^t_0 \mu_s(x)\ ds,\qquad 0\le t\le 1.
\end{equation}
It is not obvious that these are norms for $t>0$, but this fact is established
in \cite[Theorem 4.4 (ii)]{F-K}.

Since $x\in {\mathcal{O}}(y)$ means that the spectral data of $x$ and $y$ agree, we have:
\begin{lem}\label{lem5.2}
Let $M$ be a type $\text{\rm II}_1$ factor. Then the following are equivalent
for elements $x,y\in M^+$.
\begin{itemize}
\item[\rm (i)] $x\in {\mathcal{O}}(y)$.
\item[\rm (ii)] $\|x\|_{(t)} = \|y\|_{(t)}, \ 0\le t\le 1$.
\end{itemize}
\end{lem}

We now use this to investigate the Schur--Horn theorem.

Let $A$ be a masa in a type $\text{II}_1$ factor $M$, let $x\in A^+$, $0\le x\le 1$, and $z\in M^+$ be elements such that $x\preceq z$. Then define
\begin{equation}\label{eq5.11}
\Delta_{x,z} = \{y\in M\colon \ 0\le y\le 1, \  y\preceq z,\  {\bb E}_A(y) = x\},
\end{equation}
which is nonempty since it contains $x$. From the Ky Fan norm
characterization above, $\Delta_{x,z}$ is convex and it is $w^*$-compact from
\cite[Corollary 3.5]{A-M0}. By the Krein--Milman theorem, $\Delta_{x,z}$ has extreme
points. Recall the definition of $\Phi_e$ from \eqref{eq2.1}.

\begin{thm}\label{thm5.3}
Let $A$ be a masa in a type $\text{\rm II}_1$ factor $M$ and suppose that
$\Phi_e$ is noninjective for each nonzero projection $e\in M$. Let $x\in A$,
$0\le x\le 1$ and suppose that $x\preceq z$ for some element $z\in M^+$. Then
every extreme point of $\Delta_{x,z}$ lies in ${\cl O}(z)$.
\end{thm}

\begin{proof}
Fix an extreme point $b$ of $\Delta_{x,z}$. Then $\|b\|_{(t)} \leq \|z\|_{(t)}$ for
$0\leq t\leq 1$ since $b \preceq z$. To derive a contradiction, suppose that
there exists $t$ so that $\|b\|_{(t)} < \|z\|_{(t)}$, for otherwise
Lemma~\ref{lem5.2} gives the result.  Since $\|b\|_{(0)} = \|z\|_{(0)}=0$, and
$\|b\|_{(1)} = \|z\|_{(1)} = \tau(b)$, this value of $t$ lies in (0,1). The
function $\|z\|_{(t)} - \|b\|_{(t)}$ is continuous on [0,1] and so attains its
maximum value on a closed nonempty subset $\Lambda\subseteq [0,1]$. Let $t_0$ be
the least value in $\Lambda$, and let $t_1\in\Lambda$ be the largest value for
which $[t_0,t_1] \subseteq\Lambda$. Since $0,1\notin\Lambda$, we have $0<t_0\le
t_1<1$, and it is possible to have $t_0=t_1$. By continuity, there exist
$\delta_0>0$ and $\vp <\min\{t_0,1-t_1\}$ so that
\begin{equation}\label{eq:*3}
\|z\|_{(t)} - \|b\|_{(t)}>\delta_0\qquad(t\in [t_0-\vp, t_1+\vp]).
\end{equation}
On $(t_0-\vp,t_0)$, the
inequality $\mu_t(z)\le \mu_t(b)$ cannot hold everywhere because we would then
have $\|z\|_{(t_0-\vp)} - \|b\|_{(t_0-\vp)}\ge \|z\|_{(t_0)} - \|b\|_{(b_0)}$,
implying $t_0-\vp\in \Lambda$
and contradicting the minimal choice of $t_0$.
Thus there exist $\delta_1>0$ and $\vp_0\in
(0,\vp)$ so that
\begin{equation}\label{eq:*1}
\mu_{t_0-\vp_0}(z)\ge \mu_{t_0-\vp_0}(b)+\delta_1.
\end{equation}
Similarly,
if we had $\mu_t(z) \ge \mu_t(b)$ for all $t\in (t_1,t_1+\vp)$ then it would
follow that $[t_0,t_1+\vp]\subseteq \Lambda$, contradicting the maximal choice
of $t_1$. Thus there exists $\vp_1\in (0,\vp)$ so that $\mu_{t_1+\vp_1}(z)
\le\mu_{t_1+\vp_1}(b)-\delta_2$ for some $\delta_2>0$. Clearly, the values of
$\delta_0,\delta_1$ and $\delta_2$ can be replaced by their minimum value which
we denote by $\delta>0$. Thus we have the inequalities
\begin{equation}\label{eq:*33}
\|z\|_{(t)} - \|b\|_{(t)}>\delta,\qquad t\in [t_0-\vp, t_1+\vp]
\end{equation}
and
\begin{equation}\label{eq:*11}
\mu_{t_0-\vp_0}(z)\ge \mu_{t_0-\vp_0}(b)+\delta,\qquad
\mu_{t_1+\vp_1}(z)
\le\mu_{t_1+\vp_1}(b)-\delta.
\end{equation}

Now consider the interval $(t_0-\vp_0,t_1+\vp_1)$ to which we will associate a
nonzero spectral projection $e$ of $b$. There are two cases to consider. Suppose
first that $\mu_s(b)$ takes at least three distinct values on this interval.
Then there are points $\alpha_1<\alpha_2<\alpha_3 \in (t_0-\vp_0, t_1+\vp_1)$ so
that $\mu_{\alpha_1}(b) > \mu_{\alpha_2}(b) > \mu_{\alpha_3}(b)$. Then the open
interval $I = (\mu_{t_1+\vp_1}(b), \mu_{t_0-\vp_0}(b))$ contains a value
$\mu_{\alpha_2}(b)$ in the spectrum $\sigma(b)$ of $b$. Secondly, if $\mu_s(b)$
takes at most two distinct values on
$(t_0-\vp_0, t_1+\vp_1)$ then 
there exists an interval on which $\mu_s(b)$ is constant, taking a value in the interval $I$.
Now from~\eqref{eq5.5} and~\eqref{eq5.7}, we see that this value in $I$
lies in the point spectrum of $b$. In both cases the
spectral projection $e$ of $b$ for the interval $[\mu_{t_1+\vp_1}(b),
\mu_{t_0-\vp_0}(b)]$ is nonzero. 
By hypothesis there exists a nonzero self-adjoint element $w\in
eMe$ so that ${\bb E}_A(w)=0$, and by scaling we may assume that
$\|w\|<\delta/2$. Note that ${\bb E}_A(b\pm w)=x$. We now establish that $b\pm
w\in \Delta_{x,z}$ which will contradict the assumption that $b$ is an extreme
point. By symmetry we need only consider $b+w$. There are several cases.

Since $\mu_s(b)$ is nonincreasing and right continuous, there exists $r\in
[0,t_0-\vp_0]$ so that
\begin{equation}\label{eq:*2}
\mu_s(b) \le \mu_{t_0-\vp_0}(b) + \delta/2,\qquad s\in[r,t_0-\vp_0],
\end{equation}
while $\mu_s(b)>\mu_{t_0-\vp_0}(b) + \delta/2$ for $s\in [0,r)$.
 If $t\in 
[r,t_0-\vp_0]$ then
\begin{align}
\|b+w\|_{(t)} &= \|b+w\|_{(r)} + \int^t_r \mu_s(b+w)\ ds\le \|b\|_{(r)} +
\int^t_r \mu_s(b) +\delta/2 \ ds\notag\\
&\leq \|z\|_{(r)} +\int_r^t\mu_{t_0-\vp_0}(b)+\delta\ ds
\le \|z\|_{(r)}+\int_r^t\mu_{t_0-\vp_0}(z) \ ds \notag\\
\label{eq5.11a}
&\le \|z\|_{(r)} +\int^t_r \mu_s(z) \ ds = \|z\|_{(t)},
\end{align}
where we have used, respectively,
$\|w\|\le\delta/2$, \eqref{eq:*2}, \eqref{eq:*11}, and the fact that $\mu_s(z)$
is nonincreasing.

If $r=0$, then \eqref{eq5.11a} has already handled the interval
$[0,t_0-\varepsilon_0]$, so we assume that $r>0$ and we
now examine the interval $[0,r)$. Fix a value $s$ in this interval, and let $f$
be the spectral projection of
$b$ for the interval $[0,\mu_s(b)]$. By \cite[Prop. 2.2]{F-K}, $\tau(f)\geq
1-s$. Thus $\mu_s(b+\delta e/2)\leq \|(b+\delta e/2)f\|$. Since $e$ is supported on
$\sigma(b)\cap[0,\mu_s(b)-\delta/2]$, we have
$\|(b+\delta e/2)f\|=\|bf\|=\mu_s(b)$, which, together with \cite[Lemma 2.5(iii)]{F-K} implies
 that $\mu_s(b+w)\leq
\mu_s(b+\delta e/2)\leq \mu_s(b)$. It
follows by integrating these inequalities that $\|b+w\|_{(t)}\leq
\|b\|_{(t)}\leq \|z\|_{(t)}$ for $t\in
[0,r)$.

On $[t_0-\vp_0, t_1+\vp_1]$, using $\|w\|\le\delta/2$ and~\eqref{eq:*33}, we
have
\begin{equation}\label{eq5.12}
\|b+w\|_{(t)} \le \|b\|_{(t)} +\delta/2 \le \|z\|_{(t)},
\end{equation}
so it remains to consider $[t_1+\vp_1,1]$, which is handled in a similar manner
to $[0,t_0-\vp_0]$. 
Let $r'\in(t_1+\varepsilon_1,1]$ be the maximum value so that
$\mu_s(b)>\mu_{t_1+\varepsilon_1}(b)-\delta/2$ for $s\in
[t_1+\varepsilon_1, r')$. This value exists since $\mu_s(b)$ is monotone
nonincreasing and right continuous.
First consider the case $r'<1$.
Then the inequality $\mu_s(b)\leq
\mu_{t_1+\varepsilon_1}(b)-\delta/2$ holds for $s\in
[r',1]$.  
We take $s\in [r', 1]$. Then the spectral projection $g$ of $b$ for
the interval $[0,\mu_s(b)]$ is orthogonal to $e$ and has trace at least
$1-s$, by \cite[Prop. 2.2]{F-K}. Thus 
\begin{equation}\label{eqA}
\mu_s(b+w)=\|(b+\delta e/2)g\|=\|bg\|=\mu_s(b),
\end{equation}
for $s\in [r', 1]$. Thus if $t\in [r', 1]$,
then (see \cite[Prop. 2.7]{F-K})
\begin{align}\int_t^1
\mu_s(b+w)\ ds&=\int_t^1\mu_s(b)\ ds=\tau(b)-\|b\|_{(t)}\notag\\
\label{eqC}&\geq
\tau(z)-\|z\|_{(t)}=\int_t^1 \mu_s(z)\ ds.
\end{align}
Since $\tau(b+w)=\tau(b)=\tau(z)$, we have $\|b+w\|_{(t)}\leq \|z\|_{(t)}$ on
this interval. 

Let
$s\in [t_1+\varepsilon_1, r')$. Then 
\begin{align}
\mu_{s}(b+w)&\geq
\mu_s(b-\delta e/2)\geq
\mu_s(b)-\delta/2\notag\\
\label{eqB}&\geq\mu_{t_1+\varepsilon_1}(b)-\delta\geq
\mu_{t_1+\varepsilon_1}(z)\geq \mu_{s}(z).
\end{align}
If
$t\in [t_1+\varepsilon_1, r')$, then 
\begin{align}
\int_t^1
\mu_s(b+w)\ ds&=\int_t^{r'}\mu_s(b+w)\ ds+\int_{r'}^1\mu_s(b+w)\ ds\notag\\
&\geq
\int_t^{r'}\mu_s(z)\ ds+\int_{r'}^1\mu_s(z)\ ds=\int_t^1\mu_s(z)\ ds, \label{eq:5.22}
\end{align}
where we have used \eqref{eqC} with $t=r'$.
Since $\tau(b+w)=\tau(b)=\tau(z)$, we have the inequality $\|b+w\|_{(t)}\leq
\|z\|_{(t)}$
on this interval also.

This shows that $b+w\in \Delta_{x,z}$, completing the proof in the case $r'<1$.
If $r'=1$, then the proof of~\eqref{eq:5.22} for $t\in [t_1+\varepsilon_1,1)$ is exactly as
before, and this suffices to prove $b+w\in \Delta_{x,z}$.
\end{proof}

As a consequence of Theorem~\ref{thm5.3}, we can immediately deduce two
corollaries whose proofs  are so similar to those of Theorems~\ref{thm3.1},
\ref{thm3.3} and \ref{thm4.1} that we omit the details. The only minor change is
that when passing from the augmented algebras back to the original ones, instead
of observing that a certain element is countably supported we need this for a
countable set of elements, which is of course true.

\begin{cor}\label{cor5.4}
Let $A$ be either a generator masa or the radial masa in $L({\bb F}_n)$, $2\le
n<\infty$. If $x\in A^+$, $z\in L({\bb F}_n)^+$ and $x\preceq z$ then $x\in {\bb
E}_A({\cl O}(z))$.
\end{cor}

\begin{cor}\label{cor5.5}
If $A$ is the Cartan masa in the hyperfinite factor $R$ and $x\in A^+$, $z\in R^+$ satisfy $x\preceq z$, then there is a trace preserving automorphism $\theta$ of $A$ so that $\theta(x) \in {\bb E}_A({\cl O}(z))$.
\end{cor}

In \cite{A-M0}, $\{x\colon \ x\preceq z\}$ was shown to be the $\sigma$-{\em
SOT} closure
of ${\bb E}_A({\cl O}(z))$ for general masas $A$. In the case of the Cartan
masa, we
can improve this to norm closure. We will need a simple preliminary lemma.

\begin{lem}\label{lem5.6}
Let $A$ be the Cartan masa in the hyperfinite factor $R$. Given two sets of
orthogonal  projections $\{p_1,\ldots,p_n\}$ and $\{q_1,\ldots,q_n\}$ in $A$
satisfying $\tau(p_i)=\tau(q_i)$, $1\leq i\leq n$, there exists a unitary
normalizer $u$ of $A$ so that $up_iu^*=q_i$, $1\leq i\leq n$.
\end{lem}

\begin{proof}
We proceed by induction on the number $n$ of projections. The case $n=1$ is
proved in \cite{Po2} (see also \cite[Lemma 6.2.6]{SS}), so suppose that the
result is true for $n-1$ projections.
Choose a
unitary normalizer $u_1$ of $A$ so that $u_1p_1u_1^*=q_1$, and consider the
sets of projections $\{ p_1,\ldots,p_n\}$ and
$\{u_1^*q_1u_1,\ldots,u_1^*q_nu_1\}$. Since $A(1-p_1)$ is a Cartan masa in
$(1-p_1)R(1-p_1)$, we may apply the induction hypothesis to the sets of
projections
$\{ p_2,\ldots,p_n\}$ and
$\{u_1^*q_2u_1,\ldots,u_1^*q_nu_1\}$ in $A(1-p_1)$ to obtain a unitary
normalizer $w\in (1-p_1)R(1-p_1)$ of $A(1-p_1)$ so that $wp_iw^*=u_1^*q_iu_1$
for $2\leq i\leq n$. This extends to a unitary normalizer $v=w+p_1$ of $A$ in
$R$. The proof is completed by defining $u$ to be $u_1v$, so that $up_iu^*=q_i$
for $1\leq i\leq n$.
\end{proof}

\begin{cor}\label{cor5.6}
Let $A$ be the Cartan masa in the hyperfinite factor $R$. If $x\in A^+$, $z\in R^+$ and $x\preceq z$, then $x\in \ovl{{\bb E}_A({\cl O}(z))}$ (norm closure).
\end{cor}

\begin{proof}
By Corollary \ref{cor5.5} there exists a trace preserving automorphism $\theta$
of $A$ so that $\theta(x)\in {\bb E}_A({\cl O}(z))$. Given $\vp>0$, there exist
projections $p_i\in A$, $1\le i\le n$ and positive constants $\lambda_i$, $1\le
i\le n$, so that
\begin{equation}\label{eq5.16}
\left\|x - \sum^n_{i=1} \lambda_ip_i\right\|<\vp
\end{equation}
and such that $\sum\limits^n_{i=1} p_i=1$. Since $\tau(p_i)=\tau(\theta(p_i))$,
Lemma~\ref{lem5.6} gives a  unitary normalizer $u$ of $A$ satisfying
\begin{equation}\label{eq5.17}
u^*p_iu = \theta(p_i),\qquad 1\le i\le n.
\end{equation}
Choose unitaries $v_n\in R$ such that
\begin{equation}\label{eq5.18}
\lim_{n\to\infty} \|\theta(x) -{\bb E}_A(v_nzv^*_n)\|=0.
\end{equation}
From \eqref{eq5.16}, \eqref{eq5.17} and \eqref{eq5.18} it follows that
\begin{equation}\label{eq5.19}
\limsup_{n\to\infty} \|x-u {\bb E}_A(v_nzv^*_n)u^*\|\le \vp.
\end{equation}
But uniqueness of the conditional expectation gives $u{\bb E}_A(v_nzv^*_n)u^* =
{\bb E}_A(uv_nzv^*_nu^*)$, and the result follows from \eqref{eq5.19}.
\end{proof}

\noindent
{\bf Update:}
After circulating a version of this preprint, we learned of independent
work of C.\ Akemann, D.\ Sherman and N.\ Weaver that may overlap
with some of our results.

\section*{Contact Information}

\noindent Kenneth J. Dykema

\noindent Department of Mathematics, Texas A\&M University,

\noindent College Station, TX 77843

\noindent {\em E-mail address: } [Dykema]\,\,
kdykema\@@math.tamu.edu

\vspace{2ex}

\noindent Junsheng Fang

\noindent School of Mathematical Sciences,

\noindent Dalian University of Technology,

\noindent  Dalian, China

\noindent {\em E-mail address: } [Fang]\,\,
junshengfang@gmail.com

\vspace{2ex}

\noindent Donald W. Hadwin

\noindent Department of Mathematics, University of New Hampshire,

\noindent Durham NH 03824

\noindent {\em E-mail address: } [Hadwin]\,\,
don@math.unh.edu

\vspace{2ex}

\noindent Roger R. Smith

\noindent Department of Mathematics, Texas A\&M University,

\noindent College Station, TX 77843

\noindent {\em E-mail address: } [Smith]\,\,
rsmith\@@math.tamu.edu

\end{document}